\documentclass[11pt,letterpaper,twoside,reqno,nosumlimits]{amsart}

\usepackage[usenames,dvipsnames]{xcolor}
\usepackage{fancyhdr}
\usepackage{amsmath,amsfonts,amsbsy,amsgen,amscd,mathrsfs,amssymb,amsthm}
\usepackage{subfig}
\usepackage{url}

\usepackage{mathtools}

\mathtoolsset{showonlyrefs}

\usepackage[font=small,margin=0.25in,labelfont={sc},labelsep={colon}]{caption}

\usepackage{tikz}
\usepackage{microtype}
\usepackage{enumitem}

\definecolor{dark-gray}{gray}{0.3}
\definecolor{dkgray}{rgb}{.4,.4,.4}
\definecolor{dkblue}{rgb}{0,0,.5}
\definecolor{medblue}{rgb}{0,0,.75}
\definecolor{rust}{rgb}{0.5,0.1,0.1}

\usepackage[colorlinks=true]{hyperref}

\hypersetup{urlcolor=rust}
\hypersetup{citecolor=dkblue}
\hypersetup{linkcolor=dkblue}

\usepackage{setspace}

\usepackage{graphicx}
\usepackage{booktabs,longtable,tabu} 
\usepackage{multirow} 

\usepackage{float}

\usepackage[full]{textcomp}

\usepackage[scaled=.98,sups,osf]{XCharter}

\usepackage[T1]{fontenc}

\usepackage{bm}

\graphicspath{{figures/}}

\newtheorem{theorem}{Theorem}[section]

\newtheorem{proposition}[theorem]{Proposition}
\newtheorem{fact}[theorem]{Fact}

\newtheorem{corollary}[theorem]{Corollary}

\theoremstyle{definition}

\newtheorem{definition}[theorem]{Definition}
\newtheorem{example}[theorem]{Example}
\newtheorem{remark}[theorem]{Remark}

\newcommand{\term}{\emph}

\numberwithin{equation}{section} 
\numberwithin{figure}{section}
\numberwithin{table}{section}

\floatstyle{plaintop}
\newfloat{recipe}{thp}{lor}
\floatname{recipe}{Recipe}
\numberwithin{recipe}{section}

\providecommand{\mathbold}[1]{\bm{#1}}

\renewcommand{\phi}{\varphi}

\newcommand{\econst}{\mathrm{e}}

\providecommand{\mathbbm}{\mathbb} 

\newcommand{\R}{\mathbbm{R}}

\newcommand{\abs}[1]{\left\vert {#1} \right\vert}

\newcommand{\diff}[1]{\mathrm{d}{#1}}
\newcommand{\idiff}[1]{\, \diff{#1}}

\newcommand{\Prob}[1]{\mathbbm{P}\left\{{#1}\right\}}

\newcommand{\Expect}{\operatorname{\mathbb{E}}}

\DeclareMathOperator{\Var}{Var}

\newcommand{\vct}[1]{\mathbold{#1}}
\newcommand{\mtx}[1]{\mathbold{#1}}

\newcommand{\norm}[1]{\left\Vert {#1} \right\Vert}

\DeclareMathOperator{\dist}{dist}

\title[Concentration of Intrinsic Volumes]{Concentration of the Intrinsic Volumes of a Convex Body}

\author[M.~Lotz, M.~B.~McCoy, I.~Nourdin, G.~Peccati, and J.~A.~Tropp]{Martin Lotz, Michael B.~McCoy, Ivan Nourdin, Giovanni Peccati, and Joel~A.~Tropp}

\date{28 October 2018.  Revised: 19 December 2018, 15 January 2019, 24 January 2019, and 19 March 2019.}


\begin{document}

\begin{abstract}
The intrinsic volumes are measures of the content of a convex body.
This paper applies probabilistic and information-theoretic
methods to study the sequence of intrinsic volumes.
The main result states that the intrinsic volume sequence concentrates
sharply around a specific index, called the central intrinsic volume.
Furthermore, among all convex bodies whose central intrinsic volume is fixed,
an appropriately scaled cube has the intrinsic volume sequence with maximum entropy.
\end{abstract}

\maketitle

\section{Introduction and Main Results}

Intrinsic volumes are the fundamental measures of content for a convex body.
Some of the most celebrated results in convex geometry describe the properties
of the intrinsic volumes and their interrelationships.
In this paper,
we identify several new properties of the sequence of intrinsic volumes by exploiting
recent results from information theory and geometric functional analysis.
In particular, we establish that the mass of the intrinsic volume sequence
concentrates sharply around a specific index, which we call
the \term{central intrinsic volume}.  We also demonstrate that a
scaled cube has the maximum-entropy distribution of intrinsic volumes
among all convex bodies with a fixed central intrinsic volume.

\subsection{Convex Bodies and Volume}


For each natural number $m$, the Euclidean space $\R^m$
is equipped with the $\ell_2$ norm $\norm{\cdot}$,
the associated inner product, and the canonical orthonormal basis.
The origin of $\R^m$ is written as $\vct{0}_m$.

Throughout the paper, $n$ denotes a fixed natural number.
A \term{convex body} in $\R^n$ is a compact and convex subset, possibly empty.
Throughout this paper, $\mathsf{K}$ will denote a \emph{nonempty} convex body in $\R^n$. 
The dimension of the convex body, $\dim \mathsf{K}$, is the dimension of the affine hull of $\mathsf{K}$;
the dimension takes values in the range $\{0, 1, 2, \dots, n\}$.
When $\mathsf{K}$ has dimension $j$, we define the $j$-dimensional volume
$\mathrm{Vol}_j(\mathsf{K})$ to be the Lebesgue measure of $\mathsf{K}$,
computed relative to its affine hull.  If $\mathsf{K}$ is $0$-dimensional (i.e., a single point),
then $\mathrm{Vol}_0(\mathsf{K}) = 1$.

For sets $\mathsf{C} \subset \R^n$ and $\mathsf{D} \subset \R^m$,
we define the \emph{orthogonal} direct product
$$
\mathsf{C} \times \mathsf{D} := \{ (\vct{x}, \vct{y}) \in \R^{n + m} : \text{$\vct{x} \in \mathsf{C}$ and $\vct{y} \in \mathsf{D}$} \}.
$$
To be precise, the concatenation $(\vct{x}, \vct{y}) \in \R^{n + m}$ places $\vct{x} \in \R^n$
in the first $n$ coordinates and $\vct{y} \in \R^m$ in the remaining $(n - m)$ coordinates.
In particular, $\mathsf{K} \times \{ \vct{0}_m \}$
is the natural embedding of $\mathsf{K}$ into $\R^{n + m}$.

Several convex bodies merit special notation.
The unit-volume cube is the set $\mathsf{Q}_n := [0, 1]^n \subset \R^n$.
We write $\mathsf{B}_n := \{ \vct{x} \in \R^n : \norm{\vct{x}} \leq 1 \}$ for the Euclidean unit ball.
The volume $\kappa_n$ and the surface area $\omega_n$
of the Euclidean ball are given by the formulas
\begin{equation} \label{eqn:ball-vol}
\kappa_n := \mathrm{Vol}_n(\mathsf{B}_n) = \frac{\pi^{n/2}}{\Gamma(1 + n/2)}
\quad\text{and}\quad
\omega_n := n \kappa_n = \frac{2\pi^{n/2}}{\Gamma(n/2)}.
\end{equation}
As usual, $\Gamma$ denotes the gamma function.

\subsection{The Intrinsic Volumes}

In this section, we introduce the intrinsic volumes, their properties,
and connections to other geometric functionals.  A good reference
for this material is~\cite{Sch14:Convex-Bodies}.  Intrinsic volumes
are basic tools in stochastic and integral geometry~\cite{SW08:Stochastic-Integral},
and they appear in the study of random fields~\cite{AT07:Random-Fields}.

We begin with a geometrically intuitive definition. 

\begin{definition}[Intrinsic Volumes] \label{def:intvol}
For each index $j = 0, 1, 2, \dots, n$, let $\mtx{P}_j \in \R^{n \times n}$
be the orthogonal projector onto a fixed $j$-dimensional subspace of $\R^n$.
Draw a rotation matrix $\mtx{Q} \in \R^{n \times n}$
uniformly at random (from the Haar measure on the compact, homogeneous group
of $n \times n$ orthogonal matrices with determinant one).
The \emph{intrinsic volumes} of the nonempty convex body $\mathsf{K} \subset \R^n$ are the quantities
\begin{equation} \label{eqn:intvols}
V_j(\mathsf{K}) := {n \choose j} \frac{\kappa_n}{\kappa_j \kappa_{n-j}}
	\Expect_{\mtx{Q}} \big[ \mathrm{Vol}_j( \mtx{P}_j \mtx{Q} \mathsf{K} ) \big].
\end{equation}
We write $\Expect$ for expectation and $\Expect_{X}$ for expectation with respect
to a specific random variable $X$.
The intrinsic volumes of the empty set are identically zero: $V_j(\emptyset) = 0$
for each index $j$.  
\end{definition}

Up to scaling, the $j$th intrinsic volume is the average volume of a projection of the convex
body onto a $j$-dimensional subspace, chosen uniformly at random.
Following Federer~\cite{Fed59:Curvature-Measures},
we have chosen the normalization in~\eqref{eqn:intvols}
to remove the dependence on the dimension in which the
convex body is embedded.
McMullen~\cite{McM75:Nonlinear-Angle-Sum} introduced the term
``intrinsic volumes''.  In her work, Chevet~\cite{Che76:Processus-Gaussiens}
called $V_j$ the \emph{$j$-i{\`e}me {\'e}paisseur} or the ``$j$th thickness''.

\begin{example}[The Euclidean Ball] \label{ex:ball-iv}
We can easily calculate the intrinsic volumes of the Euclidean unit ball
because each projection is simply a Euclidean unit ball of lower dimension.
Thus,
$$
V_j(\mathsf{B}_n) = {n \choose j} \frac{\kappa_n}{\kappa_{n-j}}
\quad\text{for $j = 0, 1, 2, \dots, n$.}
$$
\end{example}

\begin{example}[The Cube] \label{ex:cube-iv}
We can also determine the intrinsic volumes of a cube:
$$
V_j( \mathsf{Q}_n ) = {n \choose j} 
\quad\text{for $j = 0, 1, 2, \dots, n$.}
$$
See Section~\ref{sec:parallelotope} for the details of the calculation.
A classic reference is~\cite[pp.~224--227]{San04:Integral-Geometry}.
\end{example}

\subsubsection{Geometric Functionals}

The intrinsic volumes are closely related to familiar geometric functionals.
The intrinsic volume $V_0$ is called the \emph{Euler characteristic};
it takes the value zero for the empty set and the value one for each nonempty convex body.
The intrinsic volume $V_1$ is proportional to the \emph{mean width}, 
scaled so that $V_1([0,1] \times \{ \mathbf{0}_{n-1} \}) = 1$.
Meanwhile, $V_{n-1}$ is half the surface area, and $V_n$ coincides with the
ordinary volume measure, $\mathrm{Vol}_n$.

\subsubsection{Properties}
\label{sec:iv-props}

The intrinsic volumes satisfy many important properties.  Let $\mathsf{C}, \mathsf{K} \subset \R^n$
be nonempty convex bodies.  For each index $j = 0,1,2, \dots, n$,
the intrinsic volume $V_j$ is...

\begin{enumerate}
\item	\textbf{Nonnegative:}  
$V_j(\mathsf{K}) \geq 0$.

\item	\textbf{Monotone:}  
$\mathsf{C} \subset \mathsf{K}$ implies $V_j(\mathsf{C}) \leq V_j(\mathsf{K})$.

\item	\textbf{Homogeneous:}  
$V_j(\lambda \mathsf{K}) = \lambda^j V_j(\mathsf{K})$ for each $\lambda \geq 0$.

\item	\textbf{Invariant:}  
$V_j(\mtx{T} \mathsf{K}) = V_j(\mathsf{K})$ for each \emph{proper rigid motion} $\mtx{T}$.
That is, $\mtx{T}$ acts by rotation and translation.

\item	\textbf{Intrinsic:}  
$V_j(\mathsf{K}) = V_j(\mathsf{K} \times \{ \mathbold{0}_m \})$ for each natural number $m$.

\item	\textbf{A Valuation:}  $V_j(\emptyset) = 0$.  If $\mathsf{C} \cup \mathsf{K}$ is also a convex body, then
$$
V_j( \mathsf{C} \cap \mathsf{K} ) + V_j( \mathsf{C} \cup \mathsf{K} )
	= V_j(\mathsf{C}) + V_j(\mathsf{K}).
$$

\item	\textbf{Continuous:}  If $\mathsf{K}_m \to \mathsf{K}$ in the Hausdorff metric, then
$V_j(\mathsf{K}_m) \to V_j(\mathsf{K})$.
\end{enumerate}

\noindent
With sufficient energy, one may derive all of these facts directly
from Definition~\ref{def:intvol}.
See the books~\cite{KR97:Introduction-Geometric,San04:Integral-Geometry,Gru07:Convex-Discrete,
SW08:Stochastic-Integral,Sch14:Convex-Bodies}
for further information about intrinsic volumes and related matters.

\subsubsection{Hadwiger's Characterization Theorems}

Hadwiger~\cite{Had51:Funktionalsatzes,Had52:Additive-Funktionale,Had57:Vorlesungen}
proved several wonderful theorems that characterize the intrinsic volumes.
To state these results, we need a short definition.
A valuation $F$ on $\R^n$ is \emph{simple}
if $F(\mathsf{K}) = 0$ whenever $\dim \mathsf{K} < n$.

\begin{fact}[Uniqueness of Volume]
Suppose that $F$ is a simple, invariant, continuous valuation on convex bodies in $\R^n$.
Then $F$ is a scalar multiple of the intrinsic volume $V_n$.
\end{fact}

\begin{fact}[The Basis of Intrinsic Volumes]
Suppose that $F$ is an invariant, continuous valuation on convex bodies in $\R^n$.
Then $F$ is a linear combination of the intrinsic volumes $V_0, V_1, V_2, \dots, V_n$.
\end{fact}

Together, these theorems demonstrate the fundamental importance
of intrinsic volumes in convex geometry.  They also construct a bridge to
the field of integral geometry, which provides explicit
formulas for geometric functionals defined by integrating over
geometric groups (e.g., the family of proper rigid motions).

\subsubsection{Quermassintegrals}

With a different normalization, the mean projection volume appearing in~\eqref{eqn:intvols}
is also known as a \emph{quermassintegral}. 
The relationship between the quermassintegrals and the intrinsic volumes is
\begin{equation} \label{eqn:quermassintegral}
{n \choose j} W_j^{(n)}(\mathsf{K}) := \kappa_{j} V_{n-j}(\mathsf{K})
\quad\text{for $j = 0, 1, 2, \dots, n$.}
\end{equation}
The notation reflects the fact that the quermassintegral $W^{(n)}_j$
depends on the ambient dimension $n$, while the intrinsic volume does not.

\subsection{The Intrinsic Volume Random Variable}

In view of Example~\ref{ex:cube-iv}, we see that the intrinsic volume sequence
of the cube $\mathsf{Q}_n$ is sharply peaked (around index $n/2$).
Example~\ref{ex:ball-iv} shows that intrinsic volumes of the Euclidean ball
$\mathsf{B}_n$ drop off quickly (starting around index $\sqrt{2\pi n}$).
This observation motivates us to ask whether the intrinsic volumes of
a general convex body also exhibit some type of concentration.

It is natural to apply probabilistic methods to address this question.
To that end, we first need to normalize the intrinsic volumes to construct
a probability distribution.

\begin{definition}[Normalized Intrinsic Volumes] \label{def:wills}
The \emph{total intrinsic volume} of the convex body $\mathsf{K}$,
also known as the \emph{Wills functional}~\cite{Wil73:Gitterpunktanzahl,Had75:Willssche,McM75:Nonlinear-Angle-Sum},
is the quantity
\begin{equation} \label{eqn:wills}
W(\mathsf{K}) := \sum_{j=0}^n V_j(\mathsf{K}).
\end{equation}
The \emph{normalized intrinsic volumes} compose the sequence
$$
\tilde{V}_j(\mathsf{K}) := \frac{V_j(\mathsf{K})}{W(\mathsf{K})}
\quad\text{for $j = 0, 1, 2, \dots, n$.}
$$
In particular, the sequence $\{ \tilde{V}_j(\mathsf{K}) : j = 0, 1, 2, \dots, n \}$
forms a probability distribution.
\end{definition}

\noindent
In spite of the similarity of notation, the total intrinsic volume $W$
should \emph{not} be confused with a quermassintegral.

We may now construct a random variable that reflects the distribution of
the intrinsic volumes of a convex body.

\begin{definition}[Intrinsic Volume Random Variable] \label{def:intvol-rv}
The \emph{intrinsic volume random variable} $Z_{\mathsf{K}}$
associated with a convex body $\mathsf{K}$
takes nonnegative integer values according to the distribution
\begin{equation} \label{eqn:ZK}
\Prob{ Z_{\mathsf{K}} = j } = \tilde{V}_j(\mathsf{K})
\quad\text{for $j = 0, 1, 2, \dots, n$.}
\end{equation}
\end{definition}

The mean of the intrinsic volume random variable plays a special role
in the analysis, so we exalt it with its own name and notation.

\begin{definition}[Central Intrinsic Volume]
The \emph{central intrinsic volume} of the convex body $\mathsf{K}$
is the quantity
\begin{equation} \label{eqn:centroid}
\Delta( \mathsf{K} ) := \Expect Z_{\mathsf{K}}
	= \sum_{j=0}^n j \cdot \tilde{V}_j(\mathsf{K}).
\end{equation}
Equivalently, the central intrinsic volume is the centroid of the
sequence of intrinsic volumes.
\end{definition}

Since the intrinsic volume sequence of a convex body $\mathsf{K} \subset \R^n$
is supported on $\{0, 1, 2, \dots, n\}$, it is immediate that
the central intrinsic volume satisfies $\Delta(\mathsf{K}) \in [0, n]$.
The extreme $n$ is unattainable (because a nonempty convex body
has Euler characteristic $V_0(\mathsf{K}) = 1$).
But it is easy to construct examples that achieve values across
the rest of the range.

\begin{example}[The Scaled Cube]
Fix $s \in [0, \infty)$.  Using Example~\ref{ex:cube-iv}
and the homogeneity of intrinsic volumes,
we see that total intrinsic volume of the scaled cube is
$$
W(s \mathsf{Q}_n) = \sum_{j=0}^n {n \choose j} \cdot s^j = (1+s)^n.
$$
The central intrinsic volume of the scaled cube is
$$
\Delta( s \mathsf{Q}_n ) = \frac{1}{(1+s)^{n}} \sum_{j=0}^n j \cdot {n \choose j} \cdot s^j
	= \sum_{j=0}^n j \cdot {n \choose j} \cdot \left( \frac{s}{1+s} \right)^j \left(1 - \frac{s}{1+s}\right)^{n-j}
	= \frac{ns}{1+s}.
$$
We recognize the mean of the random variable $\textsc{Bin}(s/(1+s), n)$ to reach the last identity.
Note that the quantity $\Delta(s \mathsf{Q}_n) = ns/(1+s)$
sweeps through the interval $[0, n)$ as we vary $s \in [0, \infty)$.
\end{example}

\begin{example}[Large Sets] \label{ex:large-set}
More generally, we can compute the limits of the normalized intrinsic volumes
of a growing set:
$$
\begin{aligned}
\lim_{s\to \infty} \tilde{V}_j(s \mathsf{K}) &\to 0 \quad\text{for $j < \dim \mathsf{K}$}; \\
\lim_{s\to \infty} \tilde{V}_j(s \mathsf{K}) &\to 1 \quad\text{for $j = \dim \mathsf{K}$.}
\end{aligned}
$$
This point follows from the homogeneity of intrinsic volumes, noted in Section~\ref{sec:iv-props}.
\end{example}

\subsection{Concentration of Intrinsic Volumes}

Our main result states that the intrinsic volume random variable
concentrates sharply around the central intrinsic volume.

\begin{theorem}[Concentration of Intrinsic Volumes] \label{thm:intvol-intro}
Let $\mathsf{K} \subset \R^n$ be a nonempty convex body
with intrinsic volume random variable $Z_{\mathsf{K}}$.
The variance satisfies
$$
\Var[ Z_{\mathsf{K}} ] \leq 4n.
$$
Furthermore, in the range $0 \leq t \leq \sqrt{n}$, we have the tail inequality
$$
\Prob{ \abs{Z_{\mathsf{K}} - \Expect Z_{\mathsf{K}}} \geq t \sqrt{n} }
	\leq 2 \econst^{-3t^2/28}.
$$
\end{theorem}

To prove this theorem, we first convert questions about the intrinsic volume random variable
into questions about metric geometry (Section~\ref{sec:steiner}).
We reinterpret the metric geometry formulations in terms of the information content
of a log-concave probability density.  Then we can control the variance (Section~\ref{sec:variance})
and concentration properties (Section~\ref{sec:concentration}) of the intrinsic volume random variable
using the analogous results for the information content random variable.

A general probability distribution on $\{0, 1, 2, \dots, n\}$ can have
variance higher than $n^2 / 3$.
In contrast, the intrinsic volume random variable has variance no greater than $4n$.
Moreover, the intrinsic volume random variable behaves, at worst, like a normal
random variable with mean $\Expect Z_{\mathsf{K}}$ and variance less than $5n$.
Thus, most of the mass of the intrinsic volume sequence is concentrated
on an interval of about $O( \sqrt{n} )$ indices.

Looking back to Example~\ref{ex:cube-iv}, concerning the unit-volume cube $\mathsf{Q}_n$,
we see that Theorem~\ref{thm:intvol-intro} gives a qualitatively
accurate description of the intrinsic volume sequence.  On the other
hand, the bounds for scaled cubes $s \mathsf{Q}_n$ can be quite poor;
see Section~\ref{sec:cubes}.

\subsection{Concentration of Conic Intrinsic Volumes}

Theorem~\ref{thm:intvol-intro} and its proof parallel recent developments in the theory of \emph{conic} intrinsic volumes,
which appear in the papers~\cite{ALMT14:Living-Edge,MT14:Steiner-Formulas,GNP17:Gaussian-Phase}.
Using the concentration of conic intrinsic volumes, we were able to establish
that random configurations of convex cones exhibit striking phase transitions;
these facts have applications in signal processing~\cite{McC13:Geometric-Analysis,MT14:Sharp-Recovery,ALMT14:Living-Edge,MT17:Achievable-Performance}.
We are confident that extending the ideas in the current paper
will help us discover new phase transition phenomena
in Euclidean integral geometry.

\subsection{Maximum-Entropy Convex Bodies}

The probabilistic approach to the intrinsic volume sequence
suggests other questions to investigate.
For instance, we can study the entropy of the intrinsic volume random
variable, which reflects the dispersion of the intrinsic volume sequence.

\begin{definition}[Intrinsic Entropy] \label{def:intent}
Let $\mathsf{K} \subset \R^n$ be a nonempty convex body.  The intrinsic entropy of $\mathsf{K}$
is the entropy of the intrinsic volume random variable $Z_{\mathsf{K}}$:
$$
\mathrm{IntEnt}(\mathsf{K}) := \mathrm{Ent}[Z_{\mathsf{K}}]
	= - \sum_{j=0}^n \tilde{V}_j(\mathsf{K}) \cdot \log \tilde{V}_j(\mathsf{K}).
$$
\end{definition}

We have the following extremal result.

\begin{theorem}[Cubes Have Maximum Entropy] \label{thm:maxent-intro}
Fix the ambient space $\R^n$, and let $d \in [0, n)$.
There is a scaled cube whose central intrinsic volume equals $d$:
$$
\Delta( s_{d,n} \mathsf{Q}_n ) = d
\quad\text{when}\quad
s_{d,n} = \frac{d}{n-d}.
$$
Among convex bodies with central intrinsic volume $d$,
the scaled cube $s_{d,n} \mathsf{Q}_n$ has the maximum intrinsic entropy.
Among all convex bodies, the unit-volume cube has the maximum intrinsic entropy.
In symbols,
$$
\max\{ \mathrm{IntEnt}(\mathsf{K}) : \text{$\Delta(\mathsf{K}) = d$} \}
	= \mathrm{IntEnt}(s_{d,n} \mathsf{Q}_n )
	\leq \mathrm{IntEnt}( \mathsf{Q}_n ).
$$
The maximum takes place over all nonempty convex bodies $\mathsf{K} \subset \R^n$.
\end{theorem}

The proof of Theorem~\ref{thm:maxent-intro} also depends on recent
results from information theory, as well as some deep properties
of the intrinsic volume sequence.
This analysis appears in Section~\ref{sec:maxent}.

Theorem~\ref{thm:maxent-intro} joins a long procession of results
on the extremal properties of the cube.  In particular, the cube
solves the (affine) reverse isoperimetric problem for symmetric
convex bodies~\cite{Bal91:Volume-Ratios}.
That is, every symmetric convex body $\mathsf{K} \subset \R^n$
has an affine image whose volume is one and whose surface area
is not greater than $2n$, the surface area of $\mathsf{Q}_n$.
See Section~\ref{sec:vr} for an equivalent statement.

\begin{remark}[Minimum Entropy]
The convex body consisting of a single point $\vct{x}_0 \in \R^n$
has the minimum intrinsic entropy: $\mathrm{IntEnt}(\{ \vct{x}_0 \}) = 0$.
Very large convex bodies also have negligible entropy:
$$
\lim_{s \to \infty} \mathrm{IntEnt}( s \mathsf{K} ) = 0
\quad\text{for each nonempty convex body $\mathsf{K} \subset \R^n$.}
$$
The limit is a consequence of Example~\ref{ex:large-set}.
\end{remark}

\subsection{Other Inequalities for Intrinsic Volumes}
\label{sec:ineqs}

The classic literature on convex geometry contains a number of prominent
inequalities relating the intrinsic volumes, and this topic continues to
arouse interest.  This section offers a short overview of the
main results of this type.  Our presentation is influenced
by~\cite{McM91:Inequalities-Intrinsic,PPV17:Quantitative-Reversal}.
See~\cite[Chap.~7]{Sch14:Convex-Bodies} for a comprehensive treatment.

\begin{remark}[Unrelated work]
Although the title of the paper~\cite{AS16:Whitney-Numbers} includes the phrase
``concentration of intrinsic volumes,'' the meaning is quite different.  Indeed,
the focus of that work is to study hyperplane arrangements via
the intrinsic volumes of a random sequence associated with the arrangement.
\end{remark}

\subsubsection{Ultra-Log-Concavity}

The Alexandrov--Fenchel inequality (AFI)
is a profound result on the behavior of mixed volumes; see~\cite[Sec.~7.3]{Sch14:Convex-Bodies}
or~\cite{SVH18:Mixed-Volumes}.
We can specialize the AFI from mixed volumes to the particular
case of quermassintegrals.  In this instance,
the AFI states that the quermassintegrals of a convex body
$\mathsf{K} \subset \R^n$ compose a log-concave sequence:
\begin{equation} \label{eqn:qm-lc}
W^{(n)}_j(\mathsf{K})^2 \geq W^{(n)}_{j+1}(\mathsf{K}) \cdot W^{(n)}_{j-1}(\mathsf{K})
	\quad\text{for $j = 1, 2, 3, \dots, n-1$.}
\end{equation}
As Chevet~\cite{Che76:Processus-Gaussiens} and McMullen~\cite{McM91:Inequalities-Intrinsic} independently observed,
the log-concavity~\eqref{eqn:qm-lc} of the quermassintegral sequence
implies that the intrinsic volumes form an ultra-log-concave (ULC) sequence:
\begin{equation} \label{eqn:iv-ulc}
j \cdot V_j(\mathsf{K})^2 \geq (j+1) \cdot V_{j+1}(\mathsf{K}) \cdot V_{j-1}(\mathsf{K})
	\quad\text{for $j = 1, 2, 3, \dots, n-1$.}
\end{equation}
This fact plays a key role in the proof of Theorem~\ref{thm:maxent-intro}.
For more information on log-concavity and ultra-log-concavity,
see the survey article~\cite{SW14:Log-Concavity}.

From~\eqref{eqn:iv-ulc}, Chevet and McMullen both deduce that all of the intrinsic volumes
are controlled by the first one, and they derive an estimate for the total intrinsic volume:
$$
V_j(\mathsf{K}) \leq \frac{1}{j!} V_1(\mathsf{K})^j
	\quad\text{for $j = 1, 2, 3, \dots, n$,}
	\quad\text{hence}\quad
	W(\mathsf{K}) \leq \econst^{V_1(\mathsf{K})}.
$$
This estimate implies some growth and decay properties
of the intrinsic volume sequence.  An interesting
application appears in Vitale's paper~\cite{Vit96:Wills-Functional},
which derives concentration for the supremum of a Gaussian process
from the foregoing bound on the total intrinsic volume.

It is possible to establish a concentration result for intrinsic volumes
as a direct consequence of~\eqref{eqn:iv-ulc}.  Indeed, it is intuitive that
a ULC sequence should concentrate around its centroid.  This point follows
from Caputo et al.~\cite[Sec.~3.2]{CDP09:Convex-Entropy},
which 
transcribes the usual semigroup proof of a log-Sobolev
inequality to the discrete setting.
When applied to intrinsic volumes, this method gives concentration
on the scale of the mean width $V_1(\mathsf{K})$ of the convex body $\mathsf{K}$.
This result captures a phenomenon different from
Theorem~\ref{thm:intvol-intro}, where the scale for the concentration is the dimension $n$.

\subsubsection{Isoperimetric Ratios}
\label{sec:vr}

Another classical consequence of the AFI is
a sequence of comparisons for the \emph{isoperimetric ratios}
of the volume of a convex body $\mathsf{K} \subset \R^n$,
relative to the Euclidean ball $\mathsf{B}_n$:
\begin{equation} \label{eqn:vrs}
\left( \frac{V_n(\mathsf{K})}{V_n(\mathsf{B}_n)} \right)^{1/n}
	\leq \left( \frac{V_{n-1}(\mathsf{K})}{V_{n-1}(\mathsf{B}_n)} \right)^{1/(n-1)}
	\leq \dots
	\leq \frac{V_1(\mathsf{K})}{V_1(\mathsf{B}_n)}.
\end{equation}
The first inequality is the isoperimetric inequality,
and the inequality between $V_n$ and $V_1$ is called Urysohn's inequality~\cite[Sec.~7.2]{Sch14:Convex-Bodies}.
Isoperimetric ratios play a prominent role in asymptotic convex geometry;
for example, see~\cite{Pis89:Volume-Convex,Bal97:Elementary-Introduction,AGM15:Asymptotic-Geometric}.

Some of the inequalities in~\eqref{eqn:vrs} can be inverted by applying
affine transformations.  For example, Ball's reverse isoperimetric
inequality~\cite{Bal91:Volume-Ratios} states that $\mathsf{K}$
admits an affine image $\hat{\mathsf{K}}$ for which
$$
\left( \frac{V_{n-1}(\hat{\mathsf{K}})}{V_{n-1}(\mathsf{B}_n)} \right)^{1/(n-1)}
	\leq \mathrm{const}_n \cdot \left( \frac{V_{n}(\hat{\mathsf{K}})}{V_{n}(\mathsf{B}_n)} \right)^{1/n}.
$$
The sharp value for the constant is known; equality holds when $\mathsf{K}$ is a simplex.
If we restrict our attention to symmetric convex bodies, then the cube is extremal.

The recent paper~\cite{PPV17:Quantitative-Reversal} of Paouris et al.~
contains a more complete, but less precise, set of reversals.
Suppose that $\mathsf{K}$ is a symmetric convex body.
Then there is a parameter $\beta_{\star} := \beta_{\star}(\mathsf{K})$ for which
\begin{equation} \label{eqn:paouris}
\frac{V_{1}(\mathsf{K})}{V_{1}(\mathsf{B}_n)}
	\leq \left[ 1 + \mathrm{const} \cdot \left( \beta_{\star} j \log\left(\frac{\econst}{j\beta_{\star}} \right) \right)^{1/2} \right]
	\cdot \left( \frac{V_{j}(\mathsf{K})}{V_{j}(\mathsf{B}_n)} \right)^{1/j}
\quad\text{for $j = 1, 2,3, \dots, \mathrm{const} / \beta_{\star}$.}
\end{equation}
The constants here are universal but unspecified.  This result implies
that the prefix of the sequence of isoperimetric ratios is roughly constant.
The result~\eqref{eqn:paouris} leaves open the question about the
behavior of the sequence beyond the distinguished point.

It would be interesting to reconcile the work of Paouris et al.~\cite{PPV17:Quantitative-Reversal}
with Theorem~\ref{thm:intvol-intro}.  In particular, it is unclear
whether the isoperimetric ratios remain constant,
or whether they exhibit some type of phase transition.
We believe that our techniques have implications for this question.

\section{Steiner's Formula and Distance Integrals}
\label{sec:steiner}

The first step in our program is to convert questions about the intrinsic
volume random variable into questions in metric geometry.  We can accomplish
this goal using Steiner's formula, which links the intrinsic volumes of
a convex body to its expansion properties.  We reinterpret Steiner's
formula as a distance integral, and we use this result to compute moments
of the intrinsic volume random variable.  This technique, which appears
to be novel, drives our approach.

\subsection{Steiner's Formula}

The Minkowski sum of a nonempty convex body and a Euclidean ball is
called a \emph{parallel body}.  Steiner's formula gives an explicit
expansion for the volume of the parallel body in terms of the
intrinsic volumes of the convex body.

\begin{fact}[Steiner's Formula] \label{fact:steiner}
Let $\mathsf{K} \subset \R^n$ be a nonempty convex body.  For each $\lambda \geq 0$,
$$
\mathrm{Vol}_n( \mathsf{K} + \lambda \mathsf{B}_n )
	= \sum_{j=0}^n \lambda^{n - j} \kappa_{n - j} V_j(\mathsf{K}).
$$
\end{fact}

In other words, the volume of the parallel body is a \emph{polynomial} function
of the expansion radius.  Moreover, the coefficients depend only on the intrinsic
volumes of the convex body.  The proof of Fact~\ref{fact:steiner}
is fairly easy; see~\cite{Sch14:Convex-Bodies,Gru07:Convex-Discrete}.

\begin{remark}[Steiner and Kubota]
Steiner's formula can be used to \emph{define} the intrinsic volumes.
The definition we have given in~\eqref{eqn:intvols} is usually called
\emph{Kubota's formula}; it can be derived as a consequence
of Fact~\ref{fact:steiner} and Cauchy's formula for surface area.
For example, see~\cite[Sec.~B.5]{AGM15:Asymptotic-Geometric}.
\end{remark}

\subsection{Distance Integrals}

The parallel body can also be expressed as the set of points within
a fixed distance of the convex body.  This observation motivates us
to introduce the distance to a convex set.

\begin{definition}[Distance to a Convex Body] \label{def:dist}
The distance to a nonempty convex body $\mathsf{K}$ is the function
\begin{equation*} 
\dist(\vct{x}, \mathsf{K}) := \min\big\{ \norm{ \vct{y} - \vct{x} } : {\vct{y} \in \mathsf{K}} \big\}
\quad\text{where $\vct{x} \in \R^n$.}
\end{equation*}
\end{definition}

\noindent
It is not hard to show that the distance, $\dist(\cdot, \mathsf{K})$, 
and its square, $\dist^2( \cdot, \mathsf{K} )$, are both convex functions.

Here is an alternative statement of Steiner's formula
in terms of distance integrals~\cite{Had75:Willssche}.

\begin{proposition}[Distance Integrals] \label{prop:dist-int}
Let $\mathsf{K} \subset \R^n$ be a nonempty convex body.
Let $f : \R_+ \to \R$ be an absolutely integrable function.
Provided that the integrals on the right-hand side converge,
$$
\int_{\R^n} f( \dist(\vct{x}, \mathsf{K})) \idiff{\vct{x}}
	= f(0) \cdot V_n(\mathsf{K}) 
	+ \sum_{j=0}^{n-1} \left( \omega_{n-j} \int_0^{\infty} f(r) \cdot r^{n-j-1} \idiff{r} \right)
	\cdot V_j(\mathsf{K}).
$$
This result is equivalent to Fact~\ref{fact:steiner}.
\end{proposition}

\begin{proof} 
For $r > 0$,
Steiner's formula gives an expression for the volume of
the locus of points within distance $r$ of the convex body:
$$
\mathrm{Vol}_n \{ \vct{x} \in \R^n : \dist( \vct{x}, \mathsf{K} ) \leq r \}
	= \sum_{j=0}^n r^{n - j} \kappa_{n - j} V_j(\mathsf{K}).
$$
The rate of change in this volume satisfies
\begin{equation} \label{eqn:ddr-vol}
\frac{\diff{}}{\diff{r}} \mathrm{Vol}_n \{ \vct{x} \in \R^n : \dist( \vct{x}, \mathsf{K} ) \leq r \}
	= \sum_{j=0}^{n-1} r^{n-j-1} \omega_{n-j} V_j(\mathsf{K}).
\end{equation}
We have used the relation~\eqref{eqn:ball-vol} that $\omega_{n-j} = (n-j) \kappa_{n-j}$.

Let $\mu_{\sharp}$ be the push-forward
of the Lebesgue measure on $\R^n$ to $\R_+$ by the function $\dist(\cdot; \mathsf{K})$.
That is,
$$
\mu_{\sharp}(\mathsf{A}) := \mathrm{Vol}_n \{ \vct{x} \in \R^n : \dist(\vct{x}; \mathsf{K}) \in \mathsf{A} \}
\quad\text{for each Borel set $\mathsf{A} \subset \R_+$.}
$$
This measure clearly satisfies $\mu_{\sharp}( \{ 0 \} ) = V_n(\mathsf{K})$.
Beyond that, when $0 < a < b$,
$$
\begin{aligned}
\mu_{\sharp}( (a, b] ) &= \mathrm{Vol}_n \{ \vct{x} \in \R^n : a < \dist(\vct{x}; \mathsf{K}) \leq b \}  \\
	&= \mathrm{Vol}_n \{ \vct{x} \in \R^n : \dist(\vct{x}; \mathsf{K}) \leq b \}
		- \mathrm{Vol}_n\{ \vct{x} \in \R^n : \dist(\vct{x}; \mathsf{K}) \leq a \} \\
	&= \int_a^b \frac{\diff{}}{\diff{r}} \mathrm{Vol}_n \{ \vct{x} \in \R^n : \dist(\vct{x}; \mathsf{K}) \leq r \} \idiff{r}.
\end{aligned}
$$
Therefore, by definition of the push-forward,
$$
\begin{aligned}
\int_{\R^n} f(\dist(\vct{x}; \mathsf{K})) \idiff{\vct{x}}
	&= \int_{\R_+} f(r) \idiff{\mu_{\sharp}(r)} \\
	&= f(0) \cdot V_n(\mathsf{K}) + \int_0^\infty f(r) \cdot \frac{\diff{}}{\diff{r}} \mathrm{Vol}_n \{ \vct{x} \in \R^n : \dist(\vct{x}; \mathsf{K}) \leq r \} \idiff{r}.
\end{aligned}
$$
Introduce~\eqref{eqn:ddr-vol} into the last display to arrive at the result.
\end{proof}

\subsection{Moments of the Intrinsic Volume Sequence}

We can compute moments (i.e., linear functionals) of the sequence
of intrinsic volumes by varying the function $f$ in Proposition~\ref{prop:dist-int}.
To that end, it is helpful to make another change of variables.

\begin{corollary}[Distance Integrals II] \label{cor:dist-int}
Let $\mathsf{K} \subset \R^n$ be a nonempty convex body.
Let $g : \R_+ \to \R$ be an absolutely integrable function.
Provided the integrals on the right-hand side converge,
\begin{multline*}
\int_{\R^n} g( \pi \dist^2(\vct{x}, \mathsf{K})) \cdot \econst^{-\pi \dist^2(\vct{x}, \mathsf{K})} \idiff{\vct{x}} \\
	= g(0) \cdot V_n(\mathsf{K})
	+ \sum_{j=0}^{n-1} \left(\frac{1}{\Gamma((n-j)/2)} \int_0^{\infty} g(r) \cdot r^{-1 + (n-j)/2} \econst^{-r} \idiff{r}
	\right) \cdot V_j(\mathsf{K}).
\end{multline*}
\end{corollary}

\begin{proof}
Set $f(r) = g( \pi r^2 ) \cdot \econst^{-\pi r^2}$ in Proposition~\ref{prop:dist-int}
and invoke~\eqref{eqn:ball-vol}.
\end{proof}

We are now prepared to compute some specific moments of the intrinsic volume sequence
by making special choices of $g$ in Corollary~\ref{cor:dist-int}.

\begin{example}[Total Intrinsic Volume] \label{ex:wills}
Consider the case where $g(r) = 1$. We obtain the
appealing formula
$$
\int_{\R^n} \econst^{-\pi \dist^2(\vct{x}, \mathsf{K})} \idiff{\vct{x}}
	= \sum_{j=0}^n V_j(\mathsf{K})
	= W(\mathsf{K}).
$$
The total intrinsic volume $W(\mathsf{K})$ was defined in~\eqref{eqn:wills}.
This identity appears in~\cite{Had75:Willssche,McM75:Nonlinear-Angle-Sum}. 
\end{example}

\begin{example}[Central Intrinsic Volume] \label{ex:centroid}
The choice $g(r) = 2r / W(\mathsf{K})$ yields
$$
\frac{1}{W(\mathsf{K})} \int_{\R^n} 2\pi \dist^2(\vct{x}, \mathsf{K}) \cdot \econst^{-\pi \dist^2(\vct{x}, \mathsf{K})} \idiff{\vct{x}}
	= \frac{1}{W(\mathsf{K})} \sum_{j=0}^n (n - j) \cdot V_j(\mathsf{K})
	= n - \Expect Z_{\mathsf{K}}.
$$
We have recognized the total intrinsic volume~\eqref{eqn:wills} and
the central intrinsic volume~\eqref{eqn:centroid}.
\end{example}

\begin{example}[Generating Functions] \label{ex:gf}
We can also develop an expression for the generating function
of the intrinsic volume sequence by selecting $g(r) = \econst^{(1-\lambda^2) r}$.  Thus,
\begin{equation} \label{eqn:gf}
\int_{\R^n} \econst^{-\lambda^2 \pi \dist^2(\vct{x}, \mathsf{K})} \idiff{\vct{x}}
	= \lambda^{-n} \sum_{j=0}^n \lambda^j V_j(\mathsf{K}).
\end{equation}
This expression is valid for all $\lambda > 0$.  See~\cite{Had75:Willssche}
or~\cite[Lem.~14.2.1]{SW08:Stochastic-Integral}.

We can reframe the relation~\eqref{eqn:gf} in terms of the moment generating function
of the intrinsic volume random variable $Z_{\mathsf{K}}$.  To do so, we make the change of variables
$\lambda = \econst^{\theta}$ and divide by the total intrinsic volume $W(\mathsf{K})$:
\begin{equation} \label{eqn:mgf}
\Expect \econst^{\theta (Z_{\mathsf{K}} - n)}
	= \frac{1}{W(\mathsf{K})} \int_{\R^n} \econst^{- \econst^{2\theta} \pi \dist^2(\vct{x}, \mathsf{K})} \idiff{\vct{x}}.
\end{equation}
This expression remains valid for all $\theta \in \R$.
\end{example}

\begin{remark}[Other Moments]
In fact, we can compute \emph{any} moment of the intrinsic volume sequence
by selecting an appropriate function $f$ in Proposition~\ref{prop:dist-int}.
Corollary~\ref{cor:dist-int} is designed to produce gamma integrals.
Beta integrals also arise naturally and lead to other striking relations.
For instance,
$$
\int_{\R^n} \frac{\diff{\vct{x}}}{(1 + \lambda \dist(\vct{x}, \mathsf{K}))^{n+1}} 
	= \kappa_n \lambda^{-n} \sum_{j=0}^n \lambda^{j} \frac{V_j(\mathsf{K})}{V_j(\mathsf{B}_n)}
	\quad\text{for $\lambda > 0$.}
$$
The intrinsic volumes of the Euclidean ball are computed in Example~\ref{ex:ball-iv}.
Isoperimetric ratios appear naturally in convex geometry (see Section~\ref{sec:vr}), so this type of result may have
independent interest.
\end{remark}

\section{Variance of the Intrinsic Volume Random Variable}
\label{sec:variance}

Let us embark on our study 
of the intrinsic volume random variable.
The main result of this section states that the variance of the intrinsic volume random variable
is significantly smaller than its range.  This is a more precise version of
the variance bound in Theorem~\ref{thm:intvol-intro}.

\begin{theorem}[Variance of the Intrinsic Volume Random Variable] \label{thm:intvol-var}
Let $\mathsf{K} \subset \R^n$ be a nonempty convex body with intrinsic volume
random variable $Z_{\mathsf{K}}$.  We have the inequalities
$$
\Var[ Z_{\mathsf{K}} ] 
\leq 2 (n + \Expect Z_{\mathsf{K}}) \leq 4n.
$$
\end{theorem}

The proof of Theorem~\ref{thm:intvol-var} occupies the rest of this section.
We make a connection between the distance integrals
from Section~\ref{sec:steiner} and the information content of a log-concave
probability measure.  By using recent results on the variance of information,
we can develop bounds for the distance integrals.  These results, in turn,
yield bounds on the variance of the intrinsic volume random variable.
A closely related argument, appearing in Section~\ref{sec:concentration},
produces exponential concentration.

\begin{remark}[An Alternative Argument]
Theorem~\ref{thm:intvol-var} can be sharpened using
variance inequalities for log-concave densities.  Indeed,
it holds that
$$
\Var[ Z_{\mathsf{K}} ] \leq 2( n - \Expect Z_{\mathsf{K}} ).
$$
To prove this claim, we apply the Brascamp--Lieb inequality~\cite[Thm.~4.1]{BL76:Extensions-Brunn-Minkowski}
to a perturbation of the log-concave density~\eqref{eqn:muK} described below.
It is not clear whether similar ideas lead to normal
concentration (because the density is not \emph{strongly} log-concave), 
so we have chosen to omit this development.
\end{remark}

\subsection{The Varentropy of a Log-Concave Distribution}

First, we outline some facts from information theory
about the information content in a log-concave random variable.
Let $\mu : \R^n \to \R_+$ be a log-concave probability density;
that is, a probability density that satisfies the inequalities
$$
\mu( \tau \vct{x} + (1- \tau) \vct{y}) \geq \mu(\vct{x})^{\tau} \mu(\vct{y})^{1-\tau}
\quad\text{for $\vct{x}, \vct{y} \in \R^n$ and $\tau \in [0,1]$.}
$$
We define the \emph{information content} $I_{\mu}$ of a random
point drawn from the density $\mu$ to be the random variable
\begin{equation} \label{eqn:info-cont}
I_{\mu} := - \log \mu(\vct{y})
\quad\text{where}\quad
\vct{y} \sim \mu.
\end{equation}
The symbol $\sim$ means ``has the distribution.''  The terminology
is motivated by the operational interpretation of the information content of
a \emph{discrete} random variable as the number of bits required to represent a
random realization using a code with minimal average length~\cite{BM11:Concentration-Information}.

The expected information content $\Expect I_{\mu}$ is usually known
as the \emph{entropy} of the distribution $\mu$.  The \emph{varentropy}
of the distribution is the variance of information content:
\begin{equation} \label{eqn:varent}
\mathrm{VarEnt}[\mu] := \Var[ I_{\mu} ]
	= \Expect {} ( I_{\mu} - \Expect I_{\mu} )^2.
\end{equation}
Here and elsewhere, nonlinear functions bind before the expectation.

Bobkov \& Madiman~\cite{BM11:Concentration-Information} showed that
the varentropy of a log-concave distribution on $\R^n$ is not greater than
a constant multiple of $n$.  Other researchers quickly
determined the optimal constant.  The following result was
obtained independently by Nguyen~\cite{Ngu13:Inegalites-Fonctionelles}
and by Wang~\cite{Wan14:Heat-Capacity} in their doctoral dissertations.  

\begin{fact}[Varentropy of a Log-Concave Distribution] \label{fact:varent}
Let $\mu : \R^n \to \R_+$ be a log-concave probability density.  Then
$$
\mathrm{VarEnt}[\mu] \leq n.
$$
\end{fact}

\noindent
See Fradelizi et al.~\cite{FMW16:Optimal-Concentration} for more background
and a discussion of this result.

For future reference, note that the varentropy and related quantities
exhibit a simple scale invariance.  Consider the shifted information content
$$
I_{c \mu} := - \log( c \mu( \vct{y} ) )
\quad\text{where $c > 0$ and $\vct{y} \sim \mu$.}
$$
It follows from the definition that
\begin{equation} \label{eqn:info-center}
I_{c \mu} - \Expect I_{c \mu} = I_{\mu} - \Expect I_{\mu}
\quad\text{for each $c > 0$.}
\end{equation}
In particular, $\Var[ I_{c\mu} ] = \Var[ I_{\mu} ]$.

\subsection{A Log-Concave Density}

Next, we observe that the central intrinsic volume is related to the
information content of a log-concave density.  For a nonempty convex
body $\mathsf{K} \subset \R^n$, define
\begin{equation} \label{eqn:muK}
\mu_{\mathsf{K}}(\vct{x}) := \frac{1}{W(\mathsf{K})} \econst^{- \pi \dist^2(\vct{x}, \mathsf{K})}
\quad\text{for $\vct{x} \in \R^n$.}
\end{equation}
The density $\mu_{\mathsf{K}}$ is log-concave because the squared distance
to a convex body is a convex function. 
The calculation in Example~\ref{ex:wills} ensures that $\mu_{\mathsf{K}}$ is a probability density.

Introduce the (shifted) information content random variable associated with
$\mathsf{K}$:
\begin{equation} \label{eqn:info-cont-K}
H_{\mathsf{K}} := - \log( W(\mathsf{K}) \cdot \mu_{\mathsf{K}}(\vct{y}) )
	= \pi \dist^2(\vct{y}, \mathsf{K})
	\quad\text{where}\quad
	\vct{y} \sim \mu_{\mathsf{K}}.
\end{equation}
Up to the presence of the factor $W(\mathsf{K})$,
the random variable $H_{\mathsf{K}}$ is the information content
of a random draw from the distribution $\mu_{\mathsf{K}}$.
In view of~\eqref{eqn:varent} and~\eqref{eqn:info-center},
\begin{equation} \label{eqn:var-equiv}
\Var[ H_{\mathsf{K}} ] = \Var[ I_{\mu_{\mathsf{K}}} ] = \mathrm{VarEnt}[ \mu_{\mathsf{K}} ].
\end{equation}
More generally, \emph{all} central moments and cumulants of $H_{\mathsf{K}}$
coincide with the corresponding central moments and cumulants
of $I_{\mu_{\mathsf{K}}}$:
\begin{equation} \label{eqn:central-moments}
\Expect  f( H_{\mathsf{K}} - \Expect H_{\mathsf{K}} )
	= \Expect f ( I_{\mu_{\mathsf{K}}} - \Expect I_{\mu_{\mathsf{K}}} ).
\end{equation}
This expression is valid for any function $f : \R \to \R$ such that the
expectations exist.

\subsection{Information Content and Intrinsic Volumes}

We are now prepared to connect the moments of the intrinsic volume
random variable $Z_{\mathsf{K}}$ with the moments of the information
content random variable $H_{\mathsf{K}}$.  These representations
allow us to transfer results about information content into
data about the intrinsic volumes.

Using the notation from the last section, Example~\ref{ex:centroid} gives a relation
between the expectations:
\begin{equation} \label{eqn:ZK-HK}
\Expect Z_{\mathsf{K}} = n - 2 \Expect H_{\mathsf{K}}.
\end{equation}
The next result provides a similar relationship between the variances.

\begin{proposition}[Variance of the Intrinsic Volume Random Variable] \label{prop:var}
Let $\mathsf{K} \subset \R^n$ be a nonempty convex body
with intrinsic volume random variable $Z_{\mathsf{K}}$
and information content random variable $H_{\mathsf{K}}$.
We have the variance identity
$$
\Var[ Z_{\mathsf{K}} ] = 4 \, ( \Var[ H_{\mathsf{K}} ] - \Expect H_{\mathsf{K}} ).
$$
\end{proposition}

\begin{proof}
Apply Corollary~\ref{cor:dist-int} with the function $g(r) = 4r^2 / W(\mathsf{K})$ to obtain
\begin{equation}
\begin{aligned}
4 \Expect H_{\mathsf{K}}^2 &= \frac{1}{W(\mathsf{K})} \int_{\R^n} 4\pi^2 \dist^4(\vct{x}, \mathsf{K})
	\cdot \econst^{-\pi\dist^2(\vct{x}, \mathsf{K})} \idiff{\vct{x}} \\
	&= \frac{1}{W(\mathsf{K})} \sum_{j=0}^{n-1} (n-j) ((n-j) + 2) \cdot V_j(\mathsf{K}) \\
	&= \Expect (n - Z_{\mathsf{K}})^2 + 2 \Expect[ n - Z_{\mathsf{K}} ] \\
	&= \Var[ n - Z_{\mathsf{K}} ] + ( \Expect [n - Z_{\mathsf{K}}] )^2 + 2\Expect[n - Z_{\mathsf{K}}] \\
	&= \Var[ Z_{\mathsf{K}} ] + 4 (\Expect H_{\mathsf{K}})^2 + 4 \Expect H_{\mathsf{K}}.
\end{aligned}
\end{equation}
We have used the definition~\eqref{eqn:ZK} of the intrinsic volume random variable
to express the sum as an expectation.  In the last step, we used the relation~\eqref{eqn:ZK-HK}
twice to pass to the random variable $H_{\mathsf{K}}$.  Finally, rearrange
the display to complete the proof.
\end{proof}

\subsection{Proof of Theorem~\ref{thm:intvol-var}}

We may now establish the main result of this section.
Proposition~\ref{prop:var} yields
$$
\Var[ Z_{\mathsf{K}} ]
	= 4 \, ( \Var[ H_{\mathsf{K}} ] - \Expect H_{\mathsf{K}} )
	= 4 \, \mathrm{VarEnt}[ \mu_{\mathsf{K}} ] - 2(n - \Expect Z_{\mathsf{K}}) 
	\leq 2 n + 2 \Expect Z_{\mathsf{K}} \leq 4n.
$$
We have invoked~\eqref{eqn:var-equiv} to replace the variance of $H_{\mathsf{K}}$
with the varentropy and~\eqref{eqn:ZK-HK} to replace 
$\Expect H_{\mathsf{K}}$ by the central intrinsic volume $\Expect Z_{\mathsf{K}}$.
The inequality is a consequence of Fact~\ref{fact:varent},
which controls the varentropy of the log-concave density $\mu_{\mathsf{K}}$.
We obtain the final bound by noting that $\Expect Z_{\mathsf{K}} \leq n$.

Here is an alternative approach to the final bound that highlights the
role of the varentropy:
$$
\Var[Z_{\mathsf{K}}] \leq 4 \Var[ H_{\mathsf{K}} ]
	= 4 \, \mathrm{VarEnt}[\mu_{\mathsf{K}}] \leq 4n.
$$
The first inequality follows from Proposition~\ref{prop:var},
and the second inequality is Fact~\ref{fact:varent}.

\section{Concentration of the Intrinsic Volume Random Variable}
\label{sec:concentration}

The square root of the variance of the intrinsic volume random variable
$Z_{\mathsf{K}}$ gives the scale for fluctuations about the mean.
These fluctuations have size $O(\sqrt{n})$, which is
much smaller than the $O(n)$ range of the random variable.
This observation motivates us to
investigate the concentration properties of $Z_{\mathsf{K}}$.
In this section, we develop a refined version of the tail bound
from Theorem~\ref{thm:intvol-intro}.

\begin{theorem}[Tail Bounds for Intrinsic Volumes] \label{thm:intvol-conc}
Let $\mathsf{K} \subset \R^n$ be a nonempty convex body with intrinsic volume
random variable $Z_{\mathsf{K}}$.  For all $t \geq 0$, we have the inequalities
$$
\begin{aligned}
\Prob{ Z_{\mathsf{K}} - \Expect Z_{\mathsf{K}} \geq t }
	&\leq \exp\left\{ - (n + \Expect Z_{\mathsf{K}}) \cdot \psi^*\left( \frac{t}{n + \Expect Z_{\mathsf{K}}} \right) \right\}; \\
\Prob{ Z_{\mathsf{K}} - \Expect Z_{\mathsf{K}} \leq -t }
	&\leq \exp\left\{ - (n + \Expect Z_{\mathsf{K}}) \cdot \psi^*\left( \frac{-t}{n + \Expect Z_{\mathsf{K}}} \right) \right\}.
\end{aligned}
$$
The function $\psi^*(s) := ((1+s)\log(1+s) - s)/2$ for $s > -1$.
\end{theorem}

The proof of this result follows the same pattern as the argument from Theorem~\ref{thm:intvol-var}.
In Section~\ref{sec:pf-main}, we derive Theorem~\ref{thm:intvol-conc} as an immediate consequence.

\subsection{Moment Generating Function of the Information Content}

In addition to the variance, one may study other moments of the information
content random variable.  In particular, bounds for the moment generating function (mgf) of
the centered information content lead to exponential tail bounds
for the information content.  Bobkov \& Madiman~\cite{BM11:Concentration-Information}
proved the first result in this direction.  More recently,
Fradelizi et al.~\cite{FMW16:Optimal-Concentration} have
obtained the optimal bound.

\begin{fact}[Information Content Mgf] \label{fact:conc-info}
Let $\mu : \R^n \to \R_+$ be a log-concave probability density.  For $\beta < 1$,
$$
\Expect \econst^{ \beta (I_{\mu} - \Expect I_{\mu}) }
	\leq \econst^{ n \phi(\beta) },
$$
where $\phi(s) := - s - \log(1-s)$ for $s < 1$.
The information content random variable $I_{\mu}$ is defined in~\eqref{eqn:info-cont}.
\end{fact} 

\subsection{Information Content and Intrinsic Volumes}

We extract concentration inequalities for the intrinsic volume random
variable $Z_{\mathsf{K}}$ by studying its (centered) exponential moments.
Define
$$
m_{\mathsf{K}}(\theta) := \Expect \econst^{ \theta(Z_{\mathsf{K}} - \Expect Z_{\mathsf{K}}) }
\quad\text{for $\theta \in \R$.}
$$
The first step in the argument is to represent the mgf in terms of the information
content random variable $H_{\mathsf{K}}$ defined in~\eqref{eqn:info-cont-K}.

\begin{proposition}[Mgf of Intrinsic Volume Random Variable] \label{prop:mgf}
Let $\mathsf{K} \subset \R^n$ be a nonempty convex body with intrinsic volume random variable $Z_{\mathsf{K}}$
and information content random variable $H_{\mathsf{K}}$.  For $\theta \in \R$,
$$
m_{\mathsf{K}}(\theta) = \econst^{ - \phi(\beta) \Expect H_{\mathsf{K}} } \cdot
	\Expect \econst^{ \beta (H_{\mathsf{K}} - \Expect H_{\mathsf{K}}) }
	\quad\text{where}\quad
	\beta := 1 - \econst^{2\theta}.
$$
The function $\phi$ is defined in Fact~\ref{fact:conc-info}.
\end{proposition}

\begin{proof}
The formula~\eqref{eqn:mgf} from Example~\ref{ex:gf} yields the identity
$$
\Expect \econst^{\theta (Z_{\mathsf{K}} - n)}
	= \frac{1}{W(\mathsf{K})} \int_{\R^n} \econst^{ (1 -\econst^{2\theta}) \cdot \pi \dist^2( \vct{x}, \mathsf{K} ) }
		\cdot \econst^{-\pi \dist^2( \vct{x}, \mathsf{K})} \idiff{\vct{x}}
	= \Expect \econst^{ (1 - \econst^{2\theta}) H_{\mathsf{K}} }.
$$
We can transfer this result to obtain another representation for $m_{\mathsf{K}}$. 
First, use the identity~\eqref{eqn:ZK-HK} to replace $\Expect Z_{\mathsf{K}}$
with $\Expect H_{\mathsf{K}}$. Then invoke the last display to reach
$$
\begin{aligned}
m_{\mathsf{K}}(\theta) = \Expect \econst^{\theta (Z_{\mathsf{K}} - \Expect Z_{\mathsf{K}})}
	&= \econst^{2 \theta \Expect H_{\mathsf{K}}} \Expect \econst^{\theta (Z_{\mathsf{K}} - n) } \\
	&= \econst^{2 \theta \Expect H_{\mathsf{K}}} \Expect \econst^{(1 - \econst^{2\theta}) H_{\mathsf{K}}} \\
	&= \econst^{ (1 + 2\theta - \econst^{2\theta}) \Expect H_{\mathsf{K}} } \Expect \econst^{(1 - \econst^{2\theta}) (H_{\mathsf{K}} - \Expect H_{\mathsf{K}})}
	= \econst^{ (\beta + \log(1-\beta)) \Expect H_{\mathsf{K}} } \Expect \econst^{\beta (H_{\mathsf{K}} - \Expect H_{\mathsf{K}})}.
\end{aligned}
$$
In the last step, we have made the change of variables $\beta = 1 - \econst^{2\theta}$.
Finally, identify the value $-\phi(\beta)$ in the first exponent.
\end{proof}

\subsection{A Bound for the Mgf}

We are now prepared to bound the mgf $m_{\mathsf{K}}$.  This result will lead
directly to concentration of the intrinsic volume random variable.

\begin{proposition}[A Bound for the Mgf] \label{prop:mgf-bd}
Let $\mathsf{K} \subset \R^n$ be a nonempty convex body with intrinsic volume random variable $Z_{\mathsf{K}}$.
For $\theta \in \R$,
$$
m_{\mathsf{K}}(\theta) \leq \econst^{\psi(\theta) (n + \Expect Z_{\mathsf{K}})},
$$
where $\psi(s) := (\econst^{2s} - 2 s - 1) / 2$ for $s \in \R$.
\end{proposition}

\begin{proof}
For the parameter $\beta = 1 - \econst^{2\theta}$, Proposition~\ref{prop:mgf} yields
$$
\begin{aligned}
m_{\mathsf{K}}(\theta) &= \econst^{-\phi(\beta) \Expect H_{\mathsf{K}}}
	\Expect \econst^{\beta (H_{\mathsf{K}} - \Expect H_{\mathsf{K}})} \\
	&= \econst^{-\phi(\beta) \Expect H_{\mathsf{K}}}
	\Expect \econst^{\beta (I_{\mu_{\mathsf{K}}} - \Expect I_{\mu_{\mathsf{K}}})} \\
	&\leq \econst^{-\phi(\beta)\Expect H_{\mathsf{K}}}
	\cdot \econst^{n \phi(\beta)} \\
	&= \econst^{-\phi(\beta) (n - \Expect Z_{\mathsf{K}}) / 2}
	\cdot\econst^{n \phi(\beta)} 
	= \econst^{\phi(\beta) (n + \Expect Z_{\mathsf{K}}) / 2}.
\end{aligned}
$$
To reach the second line, we use the equivalence~\eqref{eqn:central-moments}
for the central moments.  The inequality is Fact~\ref{fact:conc-info},
the mgf bound for the information content $I_{\mu_{\mathsf{K}}}$ of the log-concave density $\mu_{\mathsf{K}}$.
Afterward, we invoke~\eqref{eqn:ZK-HK} to pass from the information content
random variable $H_{\mathsf{K}}$ to the intrinsic volume random variable $Z_{\mathsf{K}}$.
The next step is algebraic.  The result follows when we return from the variable
$\beta$ to the variable $\theta$, leading to the appearance of the function $\psi$.
\end{proof}

\subsection{Proof of Theorem~\ref{thm:intvol-conc}}

The Laplace transform method,
combined with the mgf bound from Proposition~\ref{prop:mgf-bd},
produces Bennett-type inequalities for the intrinsic volume random variable.
In brief,
$$
\begin{aligned}
\Prob{ Z_{\mathsf{K}} - \Expect Z_{\mathsf{K}} \geq t }
	&\leq \inf_{\theta > 0} \econst^{-\theta t} \cdot m_{\mathsf{K}}(\theta) \\
	&\leq \inf_{\theta > 0} \econst^{-\theta t + \psi(\theta) (n + \Expect Z_{\mathsf{K}})}
	= \exp\left\{ - (n + \Expect Z_{\mathsf{K}}) \cdot \psi^*\left( \frac{t}{n + \Expect Z_{\mathsf{K}}} \right) \right\}.
\end{aligned}
$$
The Fenchel--Legendre conjugate $\psi^*$ of the function $\psi$
has the explicit form given
in the statement of Theorem~\ref{thm:intvol-conc}.
The lower tail bound follows from the same argument.

\subsection{Proof of Theorem~\ref{thm:intvol-intro}}
\label{sec:pf-main}

The concentration inequality in the main result, Theorem~\ref{thm:intvol-conc},
follows when we weaken the inequalities obtained in the last section.
Comparing derivatives, we can verify that $\psi^*(s) \geq (s^2/4)/(1 + s/3)$
for all $s > -1$.  For the interesting range, $0 \leq t \leq n$, we have
$$
\begin{aligned}
\Prob{ Z_{\mathsf{K}} - \Expect Z_{\mathsf{K}} \geq t }
	&\leq \exp\left\{ \frac{ -t^2/4}{n + \Expect Z_{\mathsf{K}} + t / 3} \right\}; \\
\Prob{ Z_{\mathsf{K}} - \Expect Z_{\mathsf{K}} \leq -t }
	&\leq \exp\left\{ \frac{ -t^2/4}{n + \Expect Z_{\mathsf{K}} - t / 3} \right\}.
\end{aligned}
$$
We may combine this pair of inequalities into a single bound:
$$
\Prob{ \abs{Z_{\mathsf{K}} - \Expect Z_{\mathsf{K}}} \geq t }
	\leq 2 \exp\left( \frac{ -t^2/4}{n + \Expect Z_{\mathsf{K}} + t / 3} \right).
$$
Make the estimate $\Expect Z_{\mathsf{K}} \leq n$, and bound the denominator
using $t \leq n$.  This completes the argument.

\section{Example: Rectangular Parallelotopes}
\label{sec:parallelotope}

In this section, we work out the intrinsic volume sequence of a rectangular parallelotope.
This computation involves the generating function of the intrinsic volume
sequence.  Because of its elegance, we develop this method in more depth
than we need to treat the example at hand.

\subsection{Generating Functions and Intrinsic Volumes}

To begin, we collect some useful information about the properties
of the generating function of the intrinsic volumes.

\begin{definition}[Intrinsic Volume Generating Function]
The \emph{generating function}
of the intrinsic volumes
of the convex body $\mathsf{K}$
is the polynomial
$$
G_{\mathsf{K}}(\lambda) := 
	\sum_{j=0}^n \lambda^j V_j(\mathsf{K})
	= W(\lambda \mathsf{K})
	\quad\text{for $\lambda > 0$.}
$$
\end{definition}

We can use the generating function to read off some information about
a convex body, including the total intrinsic volume and the central intrinsic volume.
This is a standard result~\cite[Sec.~4.1]{Wil94:Generatingfunctionology},
so we omit the elementary argument.

\begin{proposition}[Properties of the Generating Function] \label{prop:gf}
For each nonempty convex body $\mathsf{K} \subset \R^n$,
$$
W(\mathsf{K}) = G_{\mathsf{K}}(1)
\quad\text{and}\quad
\Delta(\mathsf{K}) = \frac{G_{\mathsf{K}}'(1)}{G_{\mathsf{K}}(1)}
	= (\log G_{\mathsf{K}})'(1).
$$
As usual, the prime $'$ denotes a derivative.
\end{proposition}

It is usually challenging to compute the intrinsic volumes of a convex body,
but the following fact allows us to make short work of some examples.

\begin{fact}[Direct Products] \label{fact:direct-prod}
Let $\mathsf{C} \subset \R^{n_1}$ and $\mathsf{K} \subset \R^{n_2}$
be nonempty convex bodies.
The generating function of the intrinsic volumes of the convex body
$\mathsf{C} \times \mathsf{K} \subset \R^{n_1 + n_2}$ takes the form
$$
G_{\mathsf{C} \times \mathsf{K}}(\lambda)
	= G_{\mathsf{C}}(\lambda) \cdot G_{\mathsf{K}}(\lambda).
$$
\end{fact}

For completeness, we include a short proof inspired by Hadwiger~\cite{Had75:Willssche};
see~\cite[Lem.~14.2.1]{SW08:Stochastic-Integral}.

\begin{proof}
Abbreviate $n := n_1 + n_2$.  For a point $\vct{x} \in \R^n$,
write $\vct{x} = (\vct{x}_1, \vct{x}_2)$ where $\vct{x}_i \in \R^{n_i}$.  Then
$$
\dist^2(\vct{x}, \mathsf{C} \times \mathsf{K})
	= \dist^2(\vct{x}_1, \mathsf{C}) + \dist^2(\vct{x}_2, \mathsf{K}).
$$
Invoke the formula~\eqref{eqn:gf} from Example~\ref{ex:gf} for the generating function of the intrinsic volumes
(three times!).  For $\lambda > 0$,
$$
\begin{aligned}
\lambda^{-n} \sum_{j=0}^{n} \lambda^j V_j(\mathsf{C} \times \mathsf{K})
	&= \int_{\R^{n}} \econst^{- \lambda^2 \pi \dist^2( \vct{x}, \mathsf{C} \times \mathsf{K} ) } \idiff{\vct{x}} \\
	&= \int_{\R^{n_1}} \int_{\R^{n_2}} \econst^{- \lambda^2 \pi \dist^2( \vct{x}_1, \mathsf{C} ) } \cdot
		\econst^{-\lambda^2 \pi \dist^2(\vct{x}_2, \mathsf{K}) } \idiff{\vct{x}_1} \idiff{\vct{x}_2} \\
	&=  \left( \lambda^{-n_1}\sum_{j=0}^{n_1} \lambda^j V_j(\mathsf{C}) \right)
	\left( \lambda^{-n_2}\sum_{j=0}^{n_2} \lambda^j V_j(\mathsf{K}) \right).
\end{aligned}
$$
Cancel the leading factors of $\lambda$ to complete the argument.
\end{proof}

As a corollary, we can derive an expression for the central intrinsic volume of a direct product.

\begin{corollary}[Central Intrinsic Volume of a Product] \label{cor:Delta-prod}
Let $\mathsf{C} \subset \R^{n_1}$ and $\mathsf{K} \subset \R^{n_2}$
be nonempty convex bodies.  Then
$$
\Delta( \mathsf{C} \times \mathsf{K} ) = \Delta( \mathsf{C} ) + \Delta(\mathsf{K} ).
$$
\end{corollary}

\begin{proof}
According to Proposition~\ref{prop:gf} and Fact~\ref{fact:direct-prod},
$$
\begin{aligned}
\Delta(\mathsf{C} \times \mathsf{K})
	&= (\log G_{\mathsf{C} \times \mathsf{K}})'(1)
	= (\log (G_{\mathsf{C}} G_{\mathsf{K}}))'(1) \\
	&= (\log G_{\mathsf{C}} + \log G_{\mathsf{K}})'(1)
	= (\log G_{\mathsf{C}})'(1) + (\log G_{\mathsf{K}})'(1)
	= \Delta(\mathsf{C}) + \Delta(\mathsf{K}).
\end{aligned}
$$
This is what we needed to show.
\end{proof}

\subsection{Intrinsic Volumes of a Rectangular Parallelotope}

Using Fact~\ref{fact:direct-prod}, we quickly compute the intrinsic volumes and related
statistics for a rectangular parallelotope.

\begin{proposition}[Rectangular Parallelotopes] \label{prop:parallelotopes}
For parameters $s_1, s_2, \dots, s_n \geq 0$, construct the rectangular parallelotope
$$
\mathsf{P} := [0, s_1] \times [0, s_2] \times \dots \times [0, s_n] \subset \R^n.
$$
The generating function for the intrinsic volumes of the parallelotope $\mathsf{P}$ satisfies
$$
G_{\mathsf{P}}(\lambda) 
	= \prod_{i=1}^n (1 + \lambda s_i).
$$
In particular, $V_j(\mathsf{K}) = e_j(s_1, \dots, s_n)$,
where $e_j$ denotes the $j$th elementary symmetric function.
The total intrinsic volume and central intrinsic volume satisfy
$$
W(\mathsf{P}) = \prod_{i=1}^n (1 + s_i)
\quad\text{and}\quad
\Delta(\mathsf{P}) = \sum_{i=1}^n \frac{s_i}{1+s_i}.
$$
\end{proposition}

\begin{proof}
Let $s \geq 0$.  By direct calculation from Definition~\ref{def:intvol},
the intrinsic volumes of the interval $[0,s] \subset \R^1$ are
$V_0([0,s]) = 1$ and $V_1([0,s]) = s$.  Thus,
$$
G_{[0,s]}(\lambda)
	= \sum_{j=0}^1 \lambda^j V_j( [0, s] )
	= 1 + \lambda s.
$$
Fact~\ref{fact:direct-prod} implies that
the generating function for the intrinsic volumes
of the parallelotope $\mathsf{P}$ is
$$
G_{\mathsf{P}}(\lambda) := \sum_{j=0}^n \lambda^j V_j(\mathsf{P})
	= \prod_{i=0}^n (1 + \lambda s_i ).
$$
We immediately obtain formulas for the total intrinsic volume
and the central intrinsic volume from Proposition~\ref{prop:gf}.
Alternatively, we can compute the central intrinsic volume
of an interval $[0,s]$ and use Corollary~\ref{cor:Delta-prod}
to extend this result to the parallelotope $\mathsf{P}$.
\end{proof}

\subsection{Intrinsic Volumes of a Cube}
\label{sec:cubes}

As an immediate consequence of Proposition~\ref{prop:parallelotopes},
we obtain a clean result on the intrinsic volumes of a scaled cube.

\begin{corollary}[Cubes] \label{cor:cubes}
Let $\mathsf{Q}_n \subset \R^n$ be the unit cube.  For $s \geq 0$,
the \emph{normalized} intrinsic volumes of the scaled cube $s \mathsf{Q}_n$
coincide with a binomial distribution.  For each $j = 0, 1, 2, \dots, n$,
$$
\tilde{V}_j(s \mathsf{Q}_n) = {n \choose j} \cdot p^j (1 - p)^{n-j}
\quad\text{where}\quad
p = \frac{s}{1+s}.
$$
In particular, the central intrinsic volume of the scaled cube is
$$
\Delta( s \mathsf{Q}_n ) = np = \frac{ns}{1+s}.
$$
\end{corollary}

Corollary~\ref{cor:cubes} plays a starring role in our analysis of the intrinsic volume
sequences that attain the maximum entropy.

We can also use Corollary~\ref{cor:cubes} to test
our results on the variance and concentration properties
of the intrinsic volume sequence by comparing them with exact computations
for the cube.  Fix a number $s \geq 0$, and let $p = s/(1+s)$.  Then
$$
\Var[ Z_{s \mathsf{Q}_n} ] = np(1-p) = \frac{ns}{(1+s)^2}.
$$
Meanwhile, Theorem~\ref{thm:intvol-var} gives the upper bound
$$
\Var[ Z_{s \mathsf{Q}_n} ] \leq 2 (n + np) = \frac{2n (1+2s)}{1+s}.
$$
For $s = 1$, the ratio of the upper bound to the exact variance is 12.
For $s \approx 0$ and $s \to \infty$, the ratio becomes arbitrarily large.
Similarly, Theorem~\ref{thm:intvol-conc} gives a qualitatively good description
for $s = 1$, but its predictions are far less accurate for small and large $s$.
There remains more work to do!

\section{Maximum-Entropy Distributions of Intrinsic Volumes}
\label{sec:maxent}

We have been using probabilistic methods to study the intrinsic volumes
of a convex body, and we have seen that the intrinsic volume sequence
is concentrated, as reflected in the variance bound (Theorem~\ref{thm:intvol-var})
and the exponential tail bounds (Theorem~\ref{thm:intvol-conc}).
Therefore, it is natural to consider other measures of the dispersion of the sequence.
We recall Definition~\ref{def:intent}, of the \emph{intrinsic entropy},
which is the entropy of the normalized intrinsic volume sequence.
This concept turns out to be interesting.

In this section, we will establish Theorem~\ref{thm:maxent-intro}.
This result states that, among all convex bodies with a
fixed central intrinsic volume, a scaled cube has the largest entropy.
Moreover, the unit-volume cube has the largest intrinsic entropy
among all convex bodies in a fixed dimension.
We prove this theorem using some recent observations from information theory.

\subsection{Ultra-Log-Concavity and Convex Bodies}

The key step in proving Theorem~\ref{thm:maxent-intro} is to draw
a connection between intrinsic volumes and ultra-log-concave sequences.
We begin with an important definition.

\begin{definition}[Ultra-Log-Concave Sequence]
A nonnegative sequence $\{a_j : j = 0, 1, 2, \dots \}$ is called \emph{ultra-log-concave}, briefly \emph{ULC},
if it satisfies the relations
$$
j \cdot a_j^2 \geq (j+1) \cdot a_{j+1} a_{j-1}
\quad\text{for $j = 1, 2, 3, \dots$.}
$$
It is equivalent to say that the sequence $\{ j! \, a_j : j = 0, 1, 2, \dots \}$ is log-concave.
\end{definition}

Among all finitely supported ULC probability distributions, the binomial distributions
have the maximum entropy.  This result was obtained by
Yaming Yu~\cite{Yu08:Maximum-Entropy} using methods developed
by Oliver Johnson~\cite{Joh07:Log-Concavity-Maximum} for studying
the maximum-entropy properties of Poisson distributions.

\begin{fact}[Binomial Distributions Maximize Entropy] \label{fact:maxent}
Let $p \in [0,1]$, and fix a natural number $n$.
Among all ULC probability distributions with mean $pn$
that are supported on $\{0, 1, 2, \dots, n\}$,
the binomial distribution $\textsc{Bin}(p, n)$ has the maximum entropy.
\end{fact}

These facts are relevant to our discussion because the intrinsic volumes
of a convex body form an ultra-log-concave sequence.

\begin{fact}[Intrinsic Volumes are ULC] \label{fact:ulc}
The normalized intrinsic volumes of a nonempty convex body in $\R^n$
compose a ULC probability distribution supported on $\{0, 1, 2, \dots, n\}$.
\end{fact}

This statement is a consequence of the Alexandrov--Fenchel inequalities~\cite[Sec.~7.3]{Sch14:Convex-Bodies};
see the papers of Chevet~\cite{Che76:Processus-Gaussiens} and McMullen~\cite{McM91:Inequalities-Intrinsic}.

\subsection{Proof of Theorem~\ref{thm:maxent-intro}}

With this information at hand, we quickly establish the main result of the section.
Recall that $\mathsf{Q}_n$ denotes the unit-volume cube in $\R^n$. 
Let $\mathsf{K} \subset \R^n$ be a nonempty convex body.
Define the number $p \in [0, 1)$ by the relation $pn = \Delta(\mathsf{K})$.
According to Corollary~\ref{cor:cubes},
the scaled cube $s\mathsf{Q}_n$ satisfies
$$
\Delta(s \mathsf{Q}_n) = pn = \Delta(\mathsf{K})
\quad\text{when}\quad
s = \frac{p}{1 - p}.
$$
Fact~\ref{fact:ulc} ensures that the normalized intrinsic volume sequence
of the convex body $\mathsf{K}$ is a ULC probability distribution
supported on $\{0, 1, 2, \dots, n\}$.  Since $\Expect Z_{\mathsf{K}} = \Delta(\mathsf{K}) = pn$,
Fact~\ref{fact:maxent} now delivers
$$
\mathrm{IntEnt}(\mathsf{K})
	= \mathrm{Ent}[Z_{\mathsf{K}} ]
	\leq \mathrm{Ent}[ \textsc{Bin}(p, n) ]
	= \mathrm{Ent}[ Z_{s\mathsf{Q}_n} ]
	= \mathrm{IntEnt}( s \mathsf{Q}_n ).
$$
We have used Corollary~\ref{cor:cubes} again to see
that $Z_{s \mathsf{Q}_n} \sim \textsc{Bin}(p, n)$.
The remaining identities are simply the definition of the intrinsic entropy.
In other words, the scaled cube has the maximum intrinsic entropy
among all convex bodies that share the same central intrinsic volume.

It remains to show that the unit-volume cube has maximum intrinsic
entropy among \emph{all} convex bodies.
Continuing the analysis in the last display, we find that
$$
\mathrm{IntEnt}(\mathsf{K} ) \leq \mathrm{Ent}[ \textsc{Bin}(p,n) ]
	\leq \mathrm{Ent}[ \textsc{Bin}(1/2, n) ]
	= \mathrm{Ent}[ Z_{\mathsf{Q}_n} ]
	= \mathrm{IntEnt}(\mathsf{Q}_n).
$$
Indeed, among the binomial distributions $\textsc{Bin}(p, n)$ for $p \in [0,1]$,
the maximum entropy distribution is $\textsc{Bin}(1/2, n)$.
But this is the distribution of $Z_{\mathsf{Q}_n}$,
the intrinsic volume random variable of the unit cube $\mathsf{Q}_n$.
This observation implies the remaining claim in Theorem~\ref{thm:maxent-intro}.

\section*{Acknowledgments and Affiliations}

We are grateful to Emmanuel Milman for directing us to
the literature on concentration of information.
Dennis Amelunxen, Sergey Bobkov, and Michel Ledoux
also gave feedback at an early
stage of this project.  Ramon Van Handel provided
valuable comments and citations, including the
fact 
that ULC sequences concentrate.
We thank the anonymous referee for
a careful reading and constructive remarks.

Parts of this research were completed
at Luxembourg University and
at the Institute for Mathematics and its Applications (IMA)
at the University of Minnesota.
Giovanni Peccati is supported by the internal research project STARS (R-AGR-0502-10) at Luxembourg University.
Joel A.~Tropp gratefully acknowledges
support from ONR award N00014-11-1002 and the Gordon \& Betty Moore Foundation.

Martin Lotz (\url{martin.lotz@warwick.ac.uk}) is affiliated with the Mathematics Institute, University of Warwick.
Michael McCoy (\url{mike.mccoy@getcruise.com}) is with Cruise Automation.
Ivan Nourdin (\url{ivan.nourdin@uni.lu}) and Giovanni Peccati (\url{giovanni.peccati@gmail.com}) 
are with the Unit{\'e} de Recherche en Math{\'e}matiques, University of Luxembourg.
Joel A.~Tropp (\url{jtropp@cms.caltech.edu}) is with the Department of Computing \& Mathematical Sciences,
California Institute of Technology.

\bibliographystyle{myalpha}

\end{document}